\documentclass[12pt]{article}
\usepackage{graphicx}
\usepackage{amsmath}
\usepackage{amsthm}
\usepackage{amsfonts}
\usepackage{amssymb}
\usepackage{algpseudocode}
\usepackage{fullpage}
\usepackage{url}

\usepackage[all]{xy}

\newtheorem{theorem}{Theorem}

\newtheorem{lemma}[theorem]{Lemma}

\newtheorem{case}{Case}

\title{Concerning Kurosaki's squarefree word}

\author{Serina Camungol and Narad Rampersad\\
Department of Mathematics and Statistics, University of Winnipeg\\
515 Portage Ave., Winnipeg, Manitoba, R3B 2E9 Canada\\
\texttt{\{serina.camungol, narad.rampersad\}@gmail.com}}

\begin{document}
\maketitle
\begin{abstract}
In 2008, Kurosaki gave a new construction of a (bi-)infinite
squarefree word over three letters.  We show that in fact Kurosaki's
word avoids $7/4^+$-powers, which, as shown by Dejean, is optimal over
a $3$-letter alphabet.
\end{abstract}

\section{Introduction}
In 1906, Thue \cite{Thu06} constructed an infinite word over a
$3$-letter alphabet that avoided \emph{squares}, i.e., repetitions of
the form $xx$.  In 1935, Arshon \cite{Ars35,Ars37} independently
rediscovered this and gave another construction of an infinite
squarefree word.  In 1972, Dejean \cite{Dej72} generalized Thue's
notion of repetition to \emph{fractional powers}.  She constructed an
infinite ternary word avoiding $7/4^+$-powers, i.e., repetitions of
the form $x^\alpha$, where $\alpha > 7/4$.  Furthermore, she showed
that this was optimal over a ternary alphabet.  In 2001, Klepinin and
Sukhanov \cite{KS01} examined Arshon's word in more depth and showed
that it also avoided $7/4^+$-powers.  In 2008, Kurosaki \cite{Kur08}
gave a new construction of a squarefree word.  In this short note, we
show that Kurosaki's word also avoids $7/4^+$-powers.

Let us now recall some basic definitions.  Let $\Sigma$ be a finite
alphabet and let $x$ be a word over $\Sigma$.  We denote the length of
$x$ by $|x|$.  Write $x = x_1x_2\cdots x_n$, where each $x_i \in
\Sigma$.  The word $x$ has a \emph{period} $p$ if $x_i = x_{i+p}$ for
all $i$.  The \emph{exponent} of $x$ is the quantity $|x|/p$, where
$p$ is the least period of $x$.  If $x$ has exponent $\alpha$, we say
that $x$ is an \emph{$\alpha$-power}.  A $2$-power is also called a
\emph{square}.  A word $y$ (finite or infinite) \emph{avoids
  $\alpha$-powers} (resp.~\emph{avoids $\alpha^+$-powers}) if every
factor of $y$ has exponent less than (resp.\ at most) $\alpha$.

\section{Kurosaki's word avoids $7/4^+$-powers}

We now give the definition of Kurosaki's word.  We will use the notation from his paper.
We define functions $\sigma, \rho : \{1,2,3\}^* \to \{1,2,3\}^*$ as
follows:  If $a$ is a word over $\{1,2,3\}$, then
\begin{itemize}
  \item $\sigma(a)$ is the word obtained by exchanging 1's and 2's in $a$;
  \item $\rho(a)$ is the word obtained by exchanging 2's and 3's in $a$.
\end{itemize}
We define the function $\varphi : \{ 1, 2 ,3 \}^{*} \rightarrow \{ 1,2,3 \}^{*}$ as follows:
$$ \varphi(a) = \sigma(a)a\rho(a) $$

Consider the finite words $\varphi^n(2)$, for $n =1,2,\ldots$.  These
words do not converge to a one-sided infinite word as $n$ tends to
infinity; however, since $\varphi^n(2)$ appears as the middle third of
$\varphi^{n+1}(2)$ for all $n$, we see that this sequence of words
defines, in the limit, a two-sided, or bi-infinite, word.  It is this
bi-infinite word
\[
\cdots 123213231\,213123132\,312132123 \cdots
\]
that we refer to as \emph{Kurosaki's word}.

Following Kurosaki, we write $\varphi^n(2) = \{ a^{n}_{i}
\}_{i=0}^{3^{n}-1}$.  Proofs of Theorem~\ref{squarefree},
Lemma~\ref{permutation} and Lemma~\ref{preimage} can be found in
\cite{Kur08}.

\begin{theorem}
\label{squarefree}
$\{ a^{n}_{i} \}_{i=0}^{3^{n}-1}$ is squarefree for all $n \geq 1$.
\end{theorem}

\begin{lemma}
\label{permutation}
Each triple of the form $a^{n}_{3k}a^{n}_{3k+1}a^{n}_{3k+2}$, $k = 0,
1, \ldots, 3^{n-1}-1$, is some permutation of $123$.
\end{lemma}

From now on, when we use the term \emph{triple} we mean a permutation
of $123$ occurring at a position congruent to $0$ modulo $3$ in
$\varphi^n(2)$.  Next we define the operator $f$, which extracts the
middle term of each triple: $f(\varphi^{n}(2)) = \{
a^{n}_{3i+1}\}^{3^{n-1}-1}_{i=0}$.

\begin{lemma}
\label{preimage}
The sequence $f(\varphi^{n}(2)) = \{ a^{n}_{3i+1} \}^{3^{n-1}-1}_{i=0}$  is equal to $\varphi^{n-1}(2) = \{ a^{n-1}_{i} \}^{3^{n-1}-1}_{i=0}$.
\end{lemma}

\begin{lemma}
\label{4tuple}
$\{ a^{n}_{i} \}_{i=0}^{3^{n}-1}$ has no factors of the form
$a_{3q}a_{3q+1} \cdots a_{3q+11}$ such that $a_{3q}=a_{3q+3}=a_{3q+6}=a_{3q+9}$.
\end{lemma}

\begin{proof}
For convenience we omit the superscript $n$ from the terms
$a^n_i$.  Let $a_{3q}a_{3q+1}\cdots a_{3q+11}$ be a factor of $\{
a^{n}_{i} \}_{i=0}^{3^{n}-1}$.  As $a_{3q}a_{3q+1}a_{3q+2}$,
$a_{3q+3}a_{3q+4}a_{3q+5}$, $a_{3q+6}a_{3q+7}a_{3q+8}$ and
$a_{3q+9}a_{3q+10}a_{3q+11}$ are permutations of 123
(Lemma~\ref{permutation}), it is not difficult to see that having
$a_{3q}=a_{3q+3}=a_{3q+6}=a_{3q+9}$ creates a square of length 6, 9
or 12, in contradiction to Theorem~\ref{squarefree}.
\end{proof}

We now establish the main result.

\begin{theorem}
\label{7/4powerfree}
$\{ a^n_{i} \}_{i=0}^{3^{n}-1}$ is $7/4^{+}$-power-free for all $n
\geq 1$.
\end{theorem}

\begin{proof}
We will prove that $\varphi^{n}(2) = \{ a^n_{i} \}_{i=0}^{3^{n}-1}$
does not contain a $7/4^{+}$-power of the form $xyx$, where $|x| +
|y| = s$, $|x| = r$ and $r/s>3/4$.  We prove this by induction on
$n$, considering various cases for the possible lengths of $s$.  We
first check the base cases $n = 1, 2, 3$.  Clearly
$\varphi^{1}(2)=123$ and $\varphi^{2}(2)=213 123 132$ avoid
$7/4^{+}$-powers, and by observation $\varphi^{3}(2)=123 213 231 213
123 132 312 132 123$ avoids $7/4^{+}$-powers.  So we assume
$\varphi^{n}(2)$ avoids $7/4^{+}$-powers and we show that
$\varphi^{n+1}(2)$ is $7/4^{+}$-power-free.

\begin{case}
$s \equiv 0\pmod{3}$
\end{case}
Suppose towards a contradiction that $\varphi^{n+1}(2)$ contains a
$7/4^{+}$-power of the form $xyx$, where $|x| + |y| = s$, $|x| = r$,
$r/s>3/4$ and $s \equiv 0\pmod 3$. Let us write $$xyx = a_{p}a_{p+1}
\cdots a_{p+s-1}a_{p+s}a_{p+s+1} \cdots a_{p+s+r-1},$$ where $a_{p+i}
= a_{p+s+i}$ for $0 \leq i \leq r-1$.  Again, for convenience we omit
the superscript $n$ from the terms $a^n_i$.  If $p \equiv r \equiv 0
\pmod 3$, then $f(\varphi^{n+1}(2))$ contains a $7/4^{+}$-power and by
Lemma~\ref{preimage} the word $\varphi^{n}(2)$ contains a
$7/4^{+}$-power, a contradiction to our assumption.

We will show that the repetition $xyx$ can always be extended to the
left to give a $7/4^+$-power that has the same period $s$, but now starts
at a position congruent to $0$ modulo $3$.  A similar argument shows
that the repetition can also be extended to the right so that
its ending position is congruent to $2$ modulo $3$, and therefore that
the argument of the previous paragraph can be applied to obtain a
contradiction.

If $p \equiv 1 \pmod3$, let $p=3q+1$. The first two terms of $x$
form a suffix of the triple $a_{3q}a_{3q+1}a_{3q+2}$.
Lemma~\ref{permutation} allows us to uniquely determine $a_{3q}$ from
$a_{3q+1}$ and $a_{3q+2}$.  Similarly, the values of $a_{3q+s+1}$ and
$a_{3q+s+2}$ uniquely determine the value of $a_{3q+s}$.  Since
$a_{3q+1}a_{3q+2} = a_{3q+s+1}a_{3q+s+2}$, we have $a_{3q} =
a_{3q+s}$.  Thus $a_{p-1}a_p \cdots a_{p+s+r-1}$ is a $7/4^+$-power
with period $s$ starting at position $3q$.

If $p \equiv 2 \pmod3$, let $p=3q+2$. In order to determine the two
elements preceding $a_{3q+2}$, we look at the triple succeeding
$a_{3q+2}$, namely $a_{3q+3}a_{3q+4}a_{3q+5}$.  By
Lemma~\ref{permutation}, $a_{3q}a_{3q+1}a_{3q+2}$ is some permutation
of $a_{3q+3}a_{3q+4}a_{3q+5}.$ We must have either $a_{3q+2} =
a_{3q+4}$ or $a_{3q+2} = a_{3q+5}$ (for otherwise we would have a
square, contrary to Theorem~\ref{squarefree}).  Without loss of
generality, suppose $a_{3q+2} = a_{3q+4}$; we will determine the
values of $a_{3q}$ and $a_{3q+1}$.  If $a_{3q+1} = a_{3q+3}$, then
$a_{3q+1}a_{3q+2}a_{3q+3}a_{3q+4}$ is a square, contrary to
Theorem~\ref{squarefree}, so it must be that $a_{3q+1} = a_{3q+5}$,
which implies $a_{3q} = a_{3q+3}$.  In other words, $a_{3q}a_{3q+1}$
is uniquely determined by $a_{3q+2}a_{3q+3}a_{3q+4}a_{3q+5}$.
Similarly, $a_{3q+s}a_{3q+s+1}$ is uniquely determined by
$a_{3q+s+2}a_{3q+s+3}a_{3q+s+4}a_{3q+s+5}$.  We deduce that
$a_{p-2}a_p \cdots a_{p+s+r-1}$ is a $7/4^+$-power with period $s$
starting at position $3q$.

\begin{case}
$s < 14$ and $s\not\equiv 0\pmod3$
\end{case}
As $\varphi^{n}(2)$ is $7/4^{+}$-power-free by the induction
hypothesis, it is not hard to see that $\sigma(\varphi^{n}(2))$ and
$\rho(\varphi^{n}(2))$ are also $7/4^{+}$-power-free.  Therefore, if
$\varphi^{n+1}(2)$ contains a $7/4^{+}$-power, it must be a factor of
$\sigma(\varphi^{n}(2))\varphi^{n}(2)$ or
$\varphi^{n}(2)\rho(\varphi^{n}(2))$.  Note that here we only consider
$n \geq 3$, as the smaller values of $n$ are dealt with in our base
cases, and $|\varphi^{n}(2)| \geq 27$, so if $\varphi^{n+1}(2)$
contains a $7/4^{+}$-power it cannot extend from
$\sigma(\varphi^{n}(2))$ over into $\rho(\varphi^{n}(2))$.
Furthermore, if $s<14$, then we may assume that the entire repetition
$xyx$ has length at most $\lceil 13(7/4) \rceil = 23$.

Observe that for even $n$, $\varphi^{n}(2)$ has the form
$$213 123 132 123 213 231 321 231 \cdots 312 321 312 132 123 213 123
132,$$
so 
$$\sigma(\varphi^{n}(2))\varphi^{n}(2) = \cdots 321 312 321 231 213
123 213 231 \cdot 213 123 132 123 213 231 321 231 \cdots$$
and
$$\varphi^{n}(2)\rho(\varphi^{n}(2)) = \cdots 312 321 312 132 123 213
123 132 \cdot 312 132 123 132 312 321 231 321\cdots,$$
 where the single dots represent the boundaries between
 $\sigma(\varphi^{n}(2))$, $\varphi^n(2)$ and $\rho(\varphi^{n}(2))$.

Similarly, for odd $n$, $\varphi^{n}(2)$ has the form
$$123 213 231 213 123 132 312 132 \cdots 213 231 213 123 132 312 132
123,$$
so 
$$\sigma(\varphi^{n}(2))\varphi^{n}(2) =  \cdots 123 132 123 213 231
321 231 213 \cdot 123 213 231 213 123 132 312 132 \cdots$$
and
$$\varphi^{n}(2)\rho(\varphi^{n}(2)) = \cdots 213 231 213 123 132 312 132 123 \cdot 132 312 321 312 132 123 213 123 \cdots.$$

It suffices to check the factors of length at most $23$ crossing the
indicated boundaries; by inspection we are able to conclude that
$\varphi^{n+1}(2)$ is $7/4^{+}$-power-free for $s < 14$.

\begin{case}
$s \geq 14$ and $s\not\equiv 0\pmod3$
\end{case}
This case is further divided into six subcases.  We will prove this
case for $p \equiv 1\pmod 3$, where $p$ is the starting index of our
$7/4^{+}$-power and $s \equiv 2\pmod 3$.  The proofs for the other five
cases are similar, all involving use of Lemma~\ref{4tuple}.  Suppose
towards a contradiction that $\varphi^{n+1}(2)$ contains a
$7/4^{+}$-power of the form $xyx$, where $|x| + |y| = s$, $|x| = r$
and $s \equiv 2\pmod 3$. We can write this as $a_{3q+1}a_{3q+2} \cdots
a_{3q+s}a_{3q+s+1}a_{3q+s+2} \cdots a_{3q+s+r}$, where $a_{3q+i} =
a_{3q+s+i}$ for $1 \leq i \leq r$.  Now consider the triples
$a_{3q}a_{3q+1}a_{3q+2}$ and $a_{3q+s+1}a_{3q+s+2}a_{3q+s+3}$.  As
$a_{3q+1} = a_{3q+s+1}$ and $a_{3q+2} = a_{3q+s+2}$ and both
$a_{3q}a_{3q+1}a_{3q+2}$ and $a_{3q+s+1}a_{3q+s+2}a_{3q+s+3}$ are
permutations of $123$, it must be that $a_{3q} = a_{3q+s+3}$.
Continuing in this manner we see that $a_{3q} = a_{3(q+1)} = \cdots =
a_{3(q+j)}$, where $j \leq r$ is maximal and $j \equiv 0\pmod 3$.
This is easy to see using the diagram below, where the lines signify
that the terms are equivalent.
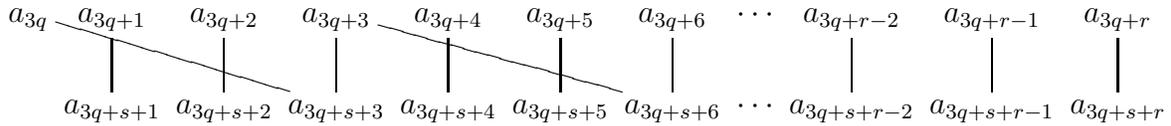
\begin{equation*}
\xymatrix@R20pt@C0pt{ a_{3q} \ar@{-}[drrr] & a_{3q+1} \ar@{-}[d] & a_{3q+2} \ar@{-}[d] & a_{3q+3}\ar@{-}[d]\ar@{-}[drrr] & a_{3q+4} \ar@{-}[d] & a_{3q+5} \ar@{-}[d] & a_{3q+6}\ar@{-}[d] & \cdots & a_{3q+r-2}\ar@{-}[d]& a_{3q+r-1}\ar@{-}[d] & a_{3q+r}\ar@{-}[d]\\
& a_{3q+s+1} &  a_{3q+s+2} & a_{3q+s+3} & a_{3q+s+4} & a_{3q+s+5} & a_{3q+s+6} & \cdots & a_{3q+s+r-2} & a_{3q+s+r-1} & a_{3q+s+r}}
\end{equation*}
However, as $s \geq 14$ it follows that $r \geq 11$ and that $j \geq
3$, so we always have $a_{3q} = a_{3(q+1)} = a_{3(q+2)} = a_{3(q+3)}$,
in contradiction to Lemma~\ref{4tuple}.

\end{proof}

\section{Conclusion}

Kurosaki also gives an alternate definition of his sequence.  In this
other formulation, the $n$-th term of the sequence is computed by a
finite automaton that reads the \emph{balanced ternary} representation
of $n$.  This representation is a base-$3$ representation with digit
set $\{-1,0,1\}$.  Kurosaki says that his sequence can be viewed as an
analogue of the classical Thue--Morse sequence for the balanced ternary
numeration system.  His sequence certainly seems to have many
interesting properties and perhaps should be better known.

\bigskip
\hrule
\bigskip
\noindent 2010 {\it Mathematics Subject Classification}: Primary
68R15.

\noindent \emph{Keywords:} repetitions, squarefree word, $7/4^+$-powers

\end{document}